\documentclass[12pt]{amsart}
\usepackage{amsmath,amssymb,amsthm}
\usepackage{geometry}                
\usepackage{graphicx}
\usepackage{amssymb}
\usepackage{epstopdf}
\usepackage{tikz}
\usepackage{pgfplots}
\usepackage{enumerate}
\usepackage{hyperref}
\usepackage{mathtools}

\def\multichoose#1#2{\ensuremath{\left(\kern-.3em\left(\genfrac{}{}{0pt}{}{#1}{#2}\right)\kern-.3em\right)}}

\usepackage{verbatim}

\hypersetup{
    colorlinks=true,
    linkcolor=blue,
    filecolor=magenta,      
    urlcolor=cyan,
    citecolor=red,
}

\newcommand{\ba}{\begin{enumerate}[(a)]}
\newcommand{\ea}{\end{enumerate}}

\newcommand{\bi}{\begin{enumerate}[(i)]}
\newcommand{\ei}{\end{enumerate}}

\definecolor{resred}{RGB}{214,0,23}

	\def\set#1{\left\{ {#1} \right\}}
	\def\setof#1#2{{\left\{#1 \mid #2\right\}}}
        \def\sdefect{\operatorname{sdefect}}

	\def\P{{\mathbb P}}

	\def\N{{\mathbb N}}

	\def\Ass{{\text{Ass}}}

\renewcommand{\geq}{\geqslant}
\renewcommand{\ge}{\geqslant}
\renewcommand{\leq}{\leqslant}
\renewcommand{\le}{\leqslant}

\theoremstyle{plain}

\newtheorem{theorem}{Theorem}[section]
\newtheorem{corollary}[theorem]{Corollary}
\newtheorem{proposition}[theorem]{Proposition}
\newtheorem{lem}[theorem]{Lemma}
\newtheorem{ex}[theorem]{Example}

\theoremstyle{definition}

\newtheorem{example}[theorem]{Example}
\newtheorem{definition}[theorem]{Definition}
\newtheorem{question}[theorem]{Question}
\newtheorem{remark}[theorem]{Remark}

\title{Comparing Powers of Edge Ideals}

\author{Mike Janssen}
\address{Mathematics/Statistics Department, Dordt College, Sioux Center, IA 51250, USA}
\email{mike.janssen@dordt.edu}

\author{Thomas Kamp}
\address{Mathematics/Statistics Department, Dordt College, Sioux Center, IA 51250, USA}
\email{thmskmp@dordt.edu}

\author{Jason Vander Woude}
\address{Department of Mathematics, University of Nebraska-Lincoln,
Lincoln, NE 68588-0130, USA}
\email{jasonvw@huskers.unl.edu}

\begin{document}

\maketitle

\begin{abstract}
	Given a nontrivial homogeneous ideal $I\subseteq k[x_1,x_2,\ldots,x_d]$, a problem of great recent interest has been the comparison of the $r$th ordinary power of $I$ and the $m$th symbolic power $I^{(m)}$.
	This comparison has been undertaken directly via an exploration of which exponents $m$ and $r$ guarantee the subset containment $I^{(m)}\subseteq I^r$ and asymptotically via a computation of the resurgence $\rho(I)$, a number for which any $m/r > \rho(I)$ guarantees $I^{(m)}\subseteq I^r$.
	Recently, a third quantity, the symbolic defect, was introduced; as $I^t\subseteq I^{(t)}$, the symbolic defect is the minimal number of generators required to add to $I^t$ in order to get $I^{(t)}$.
	
	We consider these various means of comparison when $I$ is the edge ideal of certain graphs by describing an ideal $J$ for which $I^{(t)} = I^t + J$.
	When $I$ is the edge ideal of an odd cycle, our description of the structure of $I^{(t)}$ yields solutions to both the direct and asymptotic containment questions, as well as a partial computation of the sequence of symbolic defects.
\end{abstract}



\section{Introduction}


Let $k$ be an algebraically closed field, and $I$ a nonzero proper homogeneous ideal in $S = k[x_0,x_1,x_2,\ldots,x_N]$.
Recall that the $m$th symbolic power of $I$ is the ideal
\[
	I^{(m)} = S \cap \left(\bigcap\limits_{P\in \Ass(I)} I^m S_P\right).
\]
Over the last 10--15 years, the structure of $I^{(m)}$ has been an object of ongoing study; see, e.g., the recent survey \cite{2017DaoEtAl}.
One avenue for this study has been the examination of the relationship between $I^{(m)}$ and the well-understood algebraic structure of $I^r$, the $r$th ordinary power of $I$.
The naive context in which to examine this relationship is via subset containments, i.e., for which $m$ and $r$, $s$, and $t$ do we have $I^{(m)}\subseteq I^r$ and $I^{s} \subseteq I^{(t)}$?
In fact, this line of inquiry has been extremely productive.
It is straightforward to see that $I^s \subseteq I^{(t)}$ if and only if $s \ge t$, but determining which $r$ and $m$ give $I^{(m)} \subseteq I^r$ is more delicate.

Seminal results of Ein-Lazarsfeld-Smith and Hochster-Huneke \cite{2001:einlazarsfeldsmith,2002:hochsterhuneke-invent} established that for such ideals, $I^{(m)} \subseteq I^r$ if $m/r \ge N$.
Additional information about the ideal under consideration generally leads to tighter results (see, e.g., \cite{2014:boccicooperharbourne,2013:DenkertJanssen,2013:dumnicki}).
This phenomenon led to Bocci and Harbourne's introduction of a quantity known as the \emph{resurgence} of $I$, denoted $\rho(I)$; it is the least upper bound of the set $T = \setof{m/r}{I^{(m)}\not\subseteq I^r}$.
Thus, if $m/r > \rho(I)$, we have $I^{(m)}\subseteq I^r$.

Recently, Galetto, Geramita, Shin, and Van Tuyl introduced a new measure of the difference between $I^{(m)}$ and $I^m$ known as the symbolic defect.
Since $I^m\subseteq I^{(m)}$, the quotient $I^{(m)}/I^m$ is a finite $S$-module; thus, we let $\sdefect(I,m)$ denote the number of minimal generators of $I^{(m)}/I^m$ as an $S$-module.
This is known as the symbolic defect, and the symbolic defect sequence is the sequence $\set{\sdefect(I,m)}_{m\in\N}$.
In \cite{2016:galettogeramitavantuyl}, the authors study the symbolic defect sequences of star configurations in $\P^n_k$ and homogeneous ideals of points in $\P^2_k$.

Our work considers all these questions in the context of a class of edge ideals.
Let $G = (V,E)$ be a (simple) graph on the vertex set $V = \set{x_1,x_2,\ldots,x_n}$ with edge set $E$.
The edge ideal of $G$, introduced in \cite{Villarreal1990}, is the ideal $I(G)\subseteq R = k[x_1,x_2,\ldots,x_n]$ given by
\[
	I(G) = (\setof{x_i x_j}{\set{x_i, x_j}\in E}).
\]
That is, $I(G)$ is generated by the products of pairs of those variables between which are edges in $G$.

In \cite{1994:simisvasconcelosvillareal}, the authors establish that, for an edge ideal $I = I(G)$, we have $I^{(m)} = I^m$ for all $m \ge 1$ if and only if $G$ is bipartite.
A natural question, then, is to explore the relationship between $I(G)^{(m)}$ and $I(G)^r$ when $G$ is not bipartite, which is equivalent to $G$ containing an odd cycle.
Thus, \cite{2004:worthenelliswilson} sought to explore this relationship when $G = C_{2n+1}$ is a cycle on $2n+1$ vertices.

We continue the problem of exploring the structure of the symbolic power $I(G)^{(t)}$ for certain classes of graphs $G$, with a focus on when $G$ is an odd cycle.
The main results of this work are Theorem \ref{lem:tl<it} and Corollary \ref{thm:sorta}, which together describe a decomposition of the form $I^{(t)} = I^t + J$, where $J$ is a well-understood ideal. 
We are then able to use this decomposition to resolve  \cite[Conjecture 15]{2004:worthenelliswilson}, compute $\rho(I(C_{2n+1}))$ in Theorem \ref{thm:resurgence}, and establish a partial symbolic defect sequence in Theorem \ref{thm:sdefects}.
We close by showing that our ideas in Theorem \ref{lem:tl<it} apply for complete graphs and graphs which consist of an odd cycle plus an additional vertex and edge.

\subsection*{Remark}
	As preparation of this manuscript was concluding in summer 2017, Dao et.~al posted the preprint \cite{2017DaoEtAl}.
	In particular, their Theorem 4.13 bears a striking resemblance to our Corollary \ref{cor:ordsymequal}. 
	While these similarities are worth noting, in part as evidence that interest in symbolic powers is high, it is also worth noting that the aims of these two works are distinct and complementary. 
	The aim of the relevant sections of  \cite{2017DaoEtAl} is to investigate the packing property for edge ideals, while ours is to more directly describe the difference between the ordinary and symbolic powers by investigating the structure of a set of minimal generators for $I^{(t)}$.
	We then use information about these generators to compute invariants related to the containment $I^{(m)}\subseteq I^r$.

\section{Background results}

Edge ideals are an important class of examples of squarefree monomial ideals, i.e., an ideal generated by elements of the form $x_1^{a_1} x_2^{a_2} \cdots x_n^{a_n}$, where $a_i \in \set{0,1}$ for all $i$.
When $I$ is squarefree monomial, it is well-known that the minimal primary decomposition is
\[
	I = P_1 \cap \cdots \cap P_r, \text{ with } P_j = (x_{j_1}, \ldots, x_{j_{s_j}}) \text{ for } j = 1, \ldots, r.
\]
When $I = I(G)$ is an edge ideal, the variables in the $P_j$'s are precisely the vertices in the minimal vertex covers of $G$.
Recall that, given a graph $G = (V,E)$, a \emph{vertex cover} of $G$ is a subset $V'\subseteq V$ such that for all $e\in E$, $e\cap V' \ne \emptyset$.
A \emph{minimal vertex cover} is a vertex cover minimal with respect to inclusion.

\begin{lem}[\cite{VanTuyl2013}, Corollary 3.35]\label{lem:minprimevertexcover}
	Let $G$ be a graph on the vertices $\set{x_1,x_2,\ldots,x_n}$, $I = I(G) \subseteq k[x_1,x_2,\ldots,x_n]$ be the edge ideal of $G$ and $V_1, V_2, \ldots, V_r$ the minimal vertex covers of $G$.
	Let $P_j$ be the monomial prime ideal generated by the variables in $V_j$.
	Then
	\[
		I = P_1 \cap P_2 \cap \cdots \cap P_r
	\]
	and
	\[
		I^{(m)} = P_1^m \cap P_2^m \cap \cdots \cap P_r^m.
	\]
\end{lem}

Symbolic powers of squarefree monomial ideals (and, more specifically, edge ideals) have enjoyed a great deal of recent interest (see, e.g., \cite{2013:cooperembreehahoefel,2016:BocciMFO}).
In \cite{2016:BocciMFO}, a linear programming approach is used to compute invariants related to the containment question.
We adapt this technique in Lemma \ref{lem:alphasubprog} for the edge ideals under consideration in this paper.
One result of \cite{2016:BocciMFO} which will be of use is the following. 

\begin{lem}\label{lem:jdef}
Let $I \subseteq R$ be a squarefree monomial ideal with minimal primary decomposition $I = P_1 \cap P_2 \cap \cdots \cap P_r$ with $P_j = ( x_{j_1},\ldots,x_{j_{s_j}} )$ for $j = 1,\ldots,r$. Then
$x_{a_1}\cdots x_{a_n} \in I^{(m)}$ if and only if $a_{j_1}+\ldots+a _{j_{s_j}} \geq m$ for $j=1,\ldots,r$.
\end{lem}


\begin{remark}
	Throughout this work, we will be exploring questions about ideals in $R = k[x_1,x_2,\ldots,x_n]$ related to graphs on the vertex set $\set{x_1,x_2,\ldots,x_n}$.
	We will use the $x_i$'s interchangeably to represent both vertices and variables.
	The specific use should be clear from the context, and we see this as an opportunity to emphasize the close connection between the graph and the ideal.
\end{remark}





\section{Factoring monomials along odd cycles}

In this section, we introduce the main ideas of our approach to studying symbolic powers of edge ideals.
We define a means of writing a monomial in a power of an edge ideal with respect to the minimal vertex covers of the graph and study the properties of this representation.
In what follows, let $R = k[x_1,x_2,\ldots,x_{2n+1}]$ and let $I = I(C_{2n+1})$ be the edge ideal of the odd cycle $C_{2n+1}$. 

\begin{definition}\label{defn:optimal}
Let $m\in k[x_1,x_2,\ldots,x_{2n+1}]$ be a monomial.
	Let $e_j$ denote the monomial representing the $j$th edge in the cycle, i.e., $e_j = x_j x_{j+1}$ for $1\le j \le 2n$, and $e_{2n+1} = x_{2n+1} x_1$.
	We may then write 
	$$
		m = x_1^{a_1} x_2^{a_2} \cdots x_{2n+1}^{a_{2n+1}} e_1^{b_1} e_2^{b_2} \cdots e_{2n+1}^{b_{2n+1}},
	$$
	where $b(m) := \sum b_j$ is as large as possible (observe $0\le 2b(m) \le \deg(m)$) and $a_i \ge 0$. 
	When $m$ is written in this way, we will call this an \emph{optimal factorization} of $m$, or say that $m$ is expressed in \emph{optimal form}. In addition, each $x_i^{a_i}$ with $a_i > 0$ in this form will be called an \emph{ancillary factor} of the optimal factorization, or just an \emph{ancillary} for short.
\end{definition}


Observe that the optimal form representation of $m$ is not unique in the sense that different edges may appear as factors of $m$; for example, in $k[x_1,\ldots,x_5]$, if $m = x_1^2 x_2^2 x_3 x_4 x_5$ we may write $m = x_1 e_1 e_2 e_4 = x_2 e_1 e_3 e_5$.

\begin{lem}\label{lem:splitm}
Let $m = x_1^{a_1} x_2^{a_2} \cdots x_{2n+1}^{a_{2n+1}} e_1^{b_1} e_2^{b_2} \cdots e_{2n+1}^{b_{2n+1}} \in I(C_{2n+1})$ be an optimal factorization. 
Then any $m' =  x_1^{a'_1} x_2^{a'_2} \cdots x_{2n+1}^{a'_{2n+1}} e_1^{b'_1} e_2^{b'_2} \cdots e_{2n+1}^{b'_{2n+1}}$ will also be an optimal factorization if $0 \leq a'_i \leq a_i$ and $0 \leq b'_j \leq b_j$ for all $i$ and $j$.
\end{lem}

\begin{proof}
Let $m' =  x_1^{a'_1} x_2^{a'_2} \cdots x_{2n+1}^{a'_{2n+1}} e_1^{b'_1} e_2^{b'_2} \cdots e_{2n+1}^{b'_{2n+1}}$ such that for all $i,j$, $0 \leq a'_i \leq a_i$ and $0 \leq b'_j \leq b_j$. 
Since each exponent of $m'$ is less than or equal to the corresponding exponent from $m$, we know that $m'$ divides $m$. 
Thus, there must exist some 
\[
	m'' = x_1^{(a_1-a'_1)} x_2^{(a_2-a'_2)} \cdots x_{2n+1}^{(a_{2n+1}-a'_{2n+1})} e_1^{(b_1-b'_1)} e_2^{(b_2-b'_2)} \cdots e_{2n+1}^{(b_{2n+1}-b'_{2n+1})}
\] 
such that $m = m'm''$.

Suppose that $m'$ is not in optimal form. 
Then we can re-express $m'$ as 
\[
m'=x_1^{c_1} x_2^{c_2} \cdots x_{2n+1}^{c_{2n+1}} e_1^{d_1} e_2^{d_2} \cdots e_{2n+1}^{d_{2n+1}},
\] 
such that $\sum b'_i < \sum d_i$. 
As $$m = m'm'' = \prod  x_i^{(a_i-a'_i+c_i)} e_i^{(b_i-b'_i+d_i)},$$ $m$ has an edge exponent sum of $\sum (b_i-b'_i+d_i) = \sum b_i - \sum b'_i + \sum d_i$. 
As $\sum b'_i < \sum d_i$, it must be true that $ \sum (b_i-b'_i+d_i) > \sum b_i$.

\end{proof}

The next lemma describes a process that will be critical in the proof of Theorem \ref{lem:it<tl}.
Intuitively, it says that if a monomial is factored along a path of an odd number of consecutive edges with ancillaries on both ends of this path of edges, the monomial is not written in optimal form, i.e., it can be rewritten as a product of strictly more edges.
Before stating and proving the lemma, we illustrate the process with an example.

\begin{example}\label{ex:first}
Let $G$ be a cycle with 111 vertices 
and consider $m = x_1^3 x_2^4 x_3^2 x_4^3 x_5^5 x_6^3 x_7^2 x_8^2\in I(G)$ with edge factorization: $$m = x_1  e_1^2 e_2^2 e_4^3 e_5^2 e_6 e_7 x_8$$ where  $e_i = x_i x_{i+1}$. Note that in this factorization, there is an ancillary at $x_1$ and $x_8$. We will show that $m$ is not in optimal form.

We can graphically represent $m$ by drawing an edge between $x_i$ and $x_{i+1}$ for each $e_i$ in $m$ and creating a bold outline for each ancillary, as shown below:

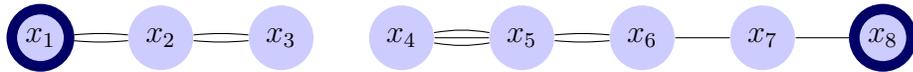
\begin{center}
\begin{tikzpicture}[scale=1.6,auto=left,every node/.style={circle,fill=blue!20,inner sep=4}]
	\node[draw=blue!40!black,line width=4,inner sep=3pt] (n1) at (0,  0) {$x_1$};
	\node (n2) at (1,  0)  {$x_2$};
	\node (n3) at (2, 0)  {$x_3$};
	\node (n4) at (3, 0) {$x_4$};
	\node (n5) at (4,  0)  {$x_5$};
	\node (n6) at (5,  0)  {$x_6$};
	\node (n7) at (6,  0)  {$x_7$};
	\node[draw=blue!40!black,line width=4,inner sep=3pt]  (n8) at (7,  0)  {$x_8$};
  \draw (n1) to [bend right = -5] (n2);
  \draw (n1) to [bend right = 5] (n2);
  \draw (n2) to [bend right = 5] (n3);
  \draw (n2) to [bend right = -5] (n3);
  \draw (n4) to [bend right = -9] (n5);
  \draw (n4) to [bend right = 0] (n5);
  \draw (n4) to [bend right = 9] (n5);
  \draw (n7) to [bend right = 0] (n6);
  \draw (n6) to [bend right = -5] (n5);
  \draw (n6) to [bend right = 5] (n5);
  \draw (n7) to [bend right = 0] (n8);
\end{tikzpicture}
\end{center}

Using a method described more fully in Lemma \ref{lem:evens}, we will ``break'' each of the red (bolded) edges back into standard $x_i$ notation so that we create new ancillaries at every vertex.

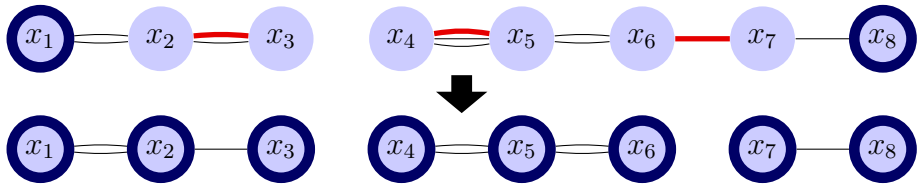
\begin{center}
\begin{tikzpicture}[scale=1.6,auto=left,every node/.style={circle,fill=blue!20,inner sep=4}]
	\node[draw=blue!40!black,line width=4,inner sep=3pt] (n1) at (0,  0) {$x_1$};
	\node (n2) at (1,  0)  {$x_2$};
	\node (n3) at (2, 0)  {$x_3$};
	\node (n4) at (3, 0) {$x_4$};
	\node (n5) at (4,  0)  {$x_5$};
	\node (n6) at (5,  0)  {$x_6$};
	\node (n7) at (6,  0)  {$x_7$};
	\node[draw=blue!40!black,line width=4,inner sep=3pt]  (n8) at (7,  0)  {$x_8$};
  \draw (n1) to [bend right = -5] (n2);
  \draw (n1) to [bend right = 5] (n2);
  \draw (n2) to [bend right = 5] (n3);
  \draw[red!90!black, line width = 2] (n2) to [bend right = -5] (n3);
  \draw[red!90!black, line width = 2] (n4) to [bend right = -9] (n5);
  \draw (n4) to [bend right = 0] (n5);
  \draw (n4) to [bend right = 9] (n5);
  \draw[red!90!black, line width = 2] (n7) to [bend right = 0] (n6);
  \draw (n6) to [bend right = -5] (n5);
  \draw (n6) to [bend right = 5] (n5);
  \draw (n7) to [bend right = 0] (n8);
\end{tikzpicture}

\begin{tikzpicture}[fill=black,ultra thick, scale = .22, transform shape,font=\Large]
\draw [fill = black] (.3,1) -- (2.7,1) -- (1.5,0) -- cycle;
\draw [fill = black] (1,1) -- (2,1) -- (2,2) -- (1,2) -- cycle;
\end{tikzpicture}

\begin{tikzpicture}[scale=1.6,auto=left,every node/.style={circle,fill=blue!20,draw=blue!40!black,line width=4,inner sep=3pt}]
	\node (n1) at (0,  0) {$x_1$};
	\node (n2) at (1,  0)  {$x_2$};
	\node (n3) at (2, 0)  {$x_3$};
	\node (n4) at (3, 0) {$x_4$};
	\node (n5) at (4,  0)  {$x_5$};
	\node (n6) at (5,  0)  {$x_6$};
	\node (n7) at (6,  0)  {$x_7$};
	\node  (n8) at (7,  0)  {$x_8$};
  \draw (n1) to [bend right = -5] (n2);
  \draw (n1) to [bend right = 5] (n2);
  \draw (n2) to [bend right = 0] (n3);
  \draw (n4) to [bend right = -5] (n5);
  \draw (n4) to [bend right = 5] (n5);
  \draw (n6) to [bend right = -5] (n5);
  \draw (n6) to [bend right = 5] (n5);
  \draw (n7) to [bend right = 0] (n8);
\end{tikzpicture}
\end{center}

Note that if we define a new monomial $p$ based on this graphical representation, where $p = x_1 x_2 x_3 x_4 x_5 x_6 x_7 x_8 e_1^2 e_2 e_4^2 e_5^2 e_7$, we have $m = p$ because we are merely changing the factorization of the monomial $m$, not its value.

As one can see, there are now 8 consecutive ancillaries, which we can pair up in a new way, as shown below. New edges are  highlighted in green (bolded in the second line).

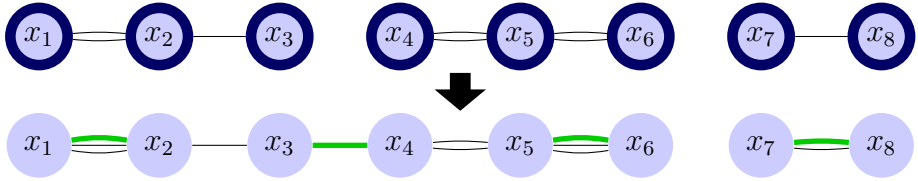
\begin{center}
\begin{tikzpicture}[scale=1.6,auto=left,every node/.style={circle,fill=blue!20,draw=blue!40!black,line width=4,inner sep=3pt}]
	\node (n1) at (0,  0) {$x_1$};
	\node (n2) at (1,  0)  {$x_2$};
	\node (n3) at (2, 0)  {$x_3$};
	\node (n4) at (3, 0) {$x_4$};
	\node (n5) at (4,  0)  {$x_5$};
	\node (n6) at (5,  0)  {$x_6$};
	\node (n7) at (6,  0)  {$x_7$};
	\node  (n8) at (7,  0)  {$x_8$};
  \draw (n1) to [bend right = -5] (n2);
  \draw (n1) to [bend right = 5] (n2);
  \draw (n2) to [bend right = 0] (n3);
  \draw (n4) to [bend right = -5] (n5);
  \draw (n4) to [bend right = 5] (n5);
  \draw (n6) to [bend right = -5] (n5);
  \draw (n6) to [bend right = 5] (n5);
  \draw (n7) to [bend right = 0] (n8);
\end{tikzpicture}

\begin{tikzpicture}[fill=black,ultra thick, scale = .22, transform shape,font=\Large]
\draw [fill = black] (.3,1) -- (2.7,1) -- (1.5,0) -- cycle;
\draw [fill = black] (1,1) -- (2,1) -- (2,2) -- (1,2) -- cycle;
\end{tikzpicture}

\begin{tikzpicture}[scale=1.6,auto=left,every node/.style={circle,fill=blue!20,inner sep=4}]
	\node (n1) at (0,  0) {$x_1$};
	\node (n2) at (1,  0)  {$x_2$};
	\node (n3) at (2, 0)  {$x_3$};
	\node (n4) at (3, 0) {$x_4$};
	\node (n5) at (4,  0)  {$x_5$};
	\node (n6) at (5,  0)  {$x_6$};
	\node (n7) at (6,  0)  {$x_7$};
	\node (n8) at (7,  0)  {$x_8$};
  \draw[green!80!black, line width = 2] (n1) to [bend right = -9] (n2);
  \draw (n1) to [bend right = 9] (n2);
  \draw (n1) to [bend right = 0] (n2);
  \draw (n2) to [bend right = 0] (n3);
  \draw[green!80!black, line width = 2] (n3) to [bend right = 0] (n4);
  \draw (n4) to [bend right = -5] (n5);
  \draw (n4) to [bend right = 5] (n5); 
  \draw[green!80!black, line width = 2] (n6) to [bend right = 9] (n5);
  \draw (n6) to [bend right = -9] (n5);
  \draw (n6) to [bend right = 0] (n5);
  \draw[green!80!black, line width = 2] (n7) to [bend right = -5] (n8);
  \draw (n7) to [bend right = 5] (n8);
\end{tikzpicture}
\end{center}

Now we have a third possible representation $q$ of this monomial. Note that $q = e_1^3 e_2 e_3 e_4^2 e_5^3 e_7^2$ and $q=p=m$. As you can see, this monomial representation has one more edge than our original representation, which means that $m$ is not optimal.
\end{example}

\begin{lem}\label{lem:evens}
Let $m = x_j^{a_j} e_j^{b_j} e_{j+1}^{b_{j+1}} \cdots e_{j+2k}^{b_{j+2k}} x_{j+2k+1}^{a_{j+2k+1}}$, where $a_j, a_{j+2k+1} \geq 1$. If it is the case that $b_{j+2h+1} \geq 1$ for all $h \in \{0,\ldots,k-1\}$, then $m$ is not in optimal form.
\end{lem}

\begin{proof}
Let $m = x_j^{a_j} e_j^{b_j} e_{j+1}^{b_{j+1}} \cdots e_{j+2k}^{b_{j+2k}} x_{j+2k+1}^{a_{j+2k+1}}$ and notice that $m$ is a string of adjacent edges with ancillaries on either end. We will show that this representation of $m$ is not optimal. 
For clarity, and without loss of generality, let $j=1$, and suppose that $b_i \geq 1$ for all evenly indexed edge exponents. 

Let $p = x_1 e_2 e_{4} \cdots e_{2k} x_{2k+2}$ and note that 
by Lemma \ref{lem:splitm}, $p$ must be in optimal form if $m$ is expressed optimally. However, 
\begin{align*}
p &= 
x_1 e_2 e_{4} \cdots e_{2k} x_{2k+2} \\
&= x_1(x_2 x_3) (x_4 x_5) \cdots (x_{2k} x_{2k+1}) x_{2k+2} \\
&= (x_1x_2) (x_3 x_4) (x_5 x_6) \cdots (x_{2k+1} x_{2k+2}) \\
&= e_1 e_3 e_5  \cdots e_{2k+1}.
\end{align*}
\end{proof}




\section{Powers of edge ideals and their structures}

We will now turn to a decomposition of $I^{(t)}$ in terms of $I^t$ and another ideal $J$ so that $I^{(t)} = I^t + J$. Our approach has numerous strengths, including the ability to easily compute the symbolic defect of $I$ for certain powers, as well as to determine which additional elements are needed to generate $I^{(t)}$ from $I^t$.


Although we will primarily focus on odd cycles in this section, we go on to show that the same underlying principles can be extended to edge ideals of other types of graphs; see Section \ref{sec:future} for more.

\begin{definition}\label{defn:weights}
Let $V'\subseteq V(G) = \set{x_1,x_2,\ldots,x_{r}}$ be a set of vertices. For a monomial $x^{\underline{a}}\in k[x_1,x_2,\ldots,x_{r}]$ with exponent vector $\underline{a} = (a_1, a_2, \ldots, a_{r})$, define the \emph{vertex weight} $w_{V'}(x^{\underline{a}})$ to be 
	\[
		w_{V'}(x^{\underline{a}}) := \sum\limits_{x_i\in V'} a_i.
	\]
\end{definition}

We will usually be interested in the case when $V'$ is a minimal vertex cover.

Using the language of vertex weights, the definition of the symbolic power of an edge ideal given in Lemma \ref{lem:jdef} becomes
\[
	I^{(t)} = (\{x^{\underline{a}} | \text{ for all minimal vertex covers } V', \, w_{V'}(x^{\underline{a}}) \geq t\}).
\]

Now define sets
\[
	L(t) = \{x^{\underline{a}} | \deg(x^{\underline{a}}) \geq 2t \text{ and for all minimal vertex covers } V', \, w_{V'}(x^{\underline{a}}) \geq t\}
	\]
and
\[
	D(t) = \{x^{\underline{a}} | \deg(x^{\underline{a}}) <2t \text{ and for all minimal vertex covers } V', \, w_{V'}(x^{\underline{a}}) \geq t\},
\]
and generate ideals $(L(t))$ and $(D(t))$, respectively.
Note that $I^{(t)} = (L(t)) + (D(t))$.
The main work of this section is to show, for the edge ideal $I$ of an odd cycle, that $I^t = (L(t))$, which is the content of Theorem \ref{lem:tl<it}.

\begin{lem}\label{lem:it sub lt}
  Let $S=k[x_1,\ldots,x_r]$, $G$ be a graph on $\set{x_1,\ldots,x_r}$, $I=I(G)$, and $L(t)$ be as defined above. Then $I^t \subseteq (L(t))$.
\end{lem}

\begin{proof}
  Suppose $m \in I^t$. 
  Write $m$ in optimal form as $m = x_1^{a_1} \cdots x_r^{a_r} \prod_{i < j} e_{ij}^{b_{ij}}$. 
  We know that given an arbitrary minimal vertex cover $V'$ and edge $e_{ij} = x_i x_j$ dividing $m$, it must be true that $x_i \in V'$ or $x_j \in V'$ or both. 
  Thus $w_{V'}(m) \geq b(m)$. 
  Further, since $m \in I^t$, we know $b(m) \geq t$ and $\deg(m) \geq 2t$, which means that $m \in (L(t))$.
\end{proof}

\begin{lem}\label{lem:m12 not in lt}
  Let $S=k[x_1,\ldots,x_r]$, $G$ be a graph on $\set{x_1,\ldots,x_r}$, $I=I(G)$, and $L(t)$ be as defined above. 
  For all $m \not\in I^t$, if $m$ has no ancillaries or a single ancillary of degree 1 then $m \not\in (L(t))$.
\end{lem}

\begin{proof}
  If there are no ancillaries in $m$ then $\deg(m) = 2 b(m) < 2t$. 
  Thus, $m$ cannot be in $L(t)$, which also means that it is not in $(L(t))$ as none of the divisors of $m$ are in $L(t)$ for a similar reason. 
  Furthermore, we reach the same conclusion if there is only one ancillary in $m$ and it has an exponent of $1$, as $\deg(m) = 2b(m)+1 < 2t+1$, and since $2b(m)+1$ and $2t+1$ are both odd, $2b(m)+1 < 2t$.
\end{proof}

For the remainder of this section, let $I=I(C_{2n+1}) \subseteq R=  k[x_1,x_2,\ldots,x_{2n+1}]$ and $V' \subseteq V(C_{2n+1})$ be a minimal vertex cover of $C_{2n+1}$.

\begin{theorem}\label{lem:it<tl}\label{lem:tl<it}
Given $I$ and $(L(t))$ as defined above, $I^t = (L(t))$.
\end{theorem}
\begin{proof}
  By Lemma \ref{lem:it sub lt} we know that $I^t \subseteq (L(t))$ so we must only show the reverse containment. Let $m\not\in I^t$, which implies that $b(m) < t$; then we will show that $m\not\in (L(t))$. Lemma \ref{lem:m12 not in lt} allows us to consider only cases where $m$ either has multiple ancillaries or has a single ancillary of at least degree 2.

  Given an arbitrary monomial $m \not \in I^t$, let $m = x_{\ell_1}^{a_{\ell_1}} x_{\ell_2}^{a_{\ell_2}} \cdots x_{\ell_r}^{a_{\ell_r}} e_1^{b_1} e_2^{b_2} \cdots e_{2n+1}^{b_{2n+1}}$ be an optimal factorization of $m$ where $x_{\ell_q}^{a_{\ell_q}}$ is an ancillary and $1 \le \ell_1 < \ell_2 < \cdots < \ell_r \le 2n+1$.


Our goal is to show that there exists some vertex cover with a weight equal to $b(m)$, and as $b(m) < t$, $m$ cannot be in $L(t)$. Since $L(t)$ is the generating set of $(L(t))$, this will be sufficient to claim that $m \not \in (L(t))$ because neither $m$, nor any of its divisors whose vertex weights can only be less than that of $m$, will be in the generating set.



We will construct a minimal vertex cover $S$ of $C_{2n+1}$ out of a sequence of subsets $S_{1}, S_{2},\ldots, S_{r}$ of $V$, where each $S_q$ is a cover for the induced subgraph $H_q$ of $C_{2n+1}$ on $$V_{H_q} = \set{x_{\ell_q},x_{\ell_{q}+1}, \ldots, x_{\ell_{q+1}-1}, x_{\ell_{q+1}}}.$$

For the sake of simplicity, let $x_i^{a_i}$ and $x_j^{a_j}$ be a pair of consecutive ancillaries (or let $x_i^{a_i}=x_{\ell_r}^{a_{\ell_r}}$ and $x_j^{a_j}=x_{\ell_1}^{a_{\ell_1}}$ in the wraparound case, or let $x_i^{a_i} = x_{\ell_1}$ and $x_j^{a_j} = x_{\ell_1}^{a_{\ell_1-1}}$ in the case of a single ancillary with degree greater than 1).
In addition, let $m_q = x_i^{a_i} e_{i}^{b_i} e_{i+1}^{b_{i+1}} \cdots e_{j-1}^{b_{j-1}} x_{j}^{a_{j}}$. Note that by Lemma \ref{lem:splitm}, $m_q$ is in optimal form.

We will show for each subgraph ${H_q}$, there exists some set of vertices $S_q \subseteq V_{H_q}$ that covers ${H_q}$ such that $w_{S_q}(m_q) = b(m_q)$.

\textbf{Case 1:} Suppose that $V_{H_q}$ has an odd number of elements. Consider $$S_q = \{x_{i+1}, x_{i+3},\ldots,x_{j-1}\}.$$ We claim that $w_{S_q}(m_q) = b(m_q)$.  This can be shown as follows:

\begin{align*}
m_q &= x_i^{a_i} e_{i}^{b_i} e_{i+1}^{b_{i+1}} \cdots e_{j-2}^{b_{j-2}}e_{j-1}^{b_{j-1}} x_{j}^{a_{j}}\\
&= x_i^{a_i} (x_{i} x_{i+1})^{b_i} (x_{i+1} x_{i+2})^{b_{i+1}} \cdots (x_{j-2}x_{j-1})^{b_{j-2}}(x_{j-1} x_j)^{b_{j-1}} x_{j}^{a_{j}}\\
&= x_i^{(a_i + b_i)} x_{i+1}^{(b_{i} + b_{i+1})} x_{i+2}^{(b_{i+1} + b_{i+2})} \cdots x_{j-1}^{(b_{j-2} + b_{j-1})} x_{j}^{(b_{j-1} + a_{j})}.
\end{align*}


By Definition \ref{defn:weights}, 

\begin{align*}
w_{S_q}(m_q) &= (b_i+b_{i+1}) + (b_{i+2}+b_{i+3}) + \cdots + (b_{j-2}+b_{j-1})\\
&= \sum\limits_{h = i}^{j-1} b_h\\
&= b(m_q).
\end{align*}

\textbf{Case 2:} Suppose now that $V_{H_q}$ has an even number of elements. If $V_{H_q} = \set{x_i,x_j}$, then the two ancillaries are adjacent and $m$ is not in optimal form, so we know that $V_{H_q}$ contains additional vertices. 
Moreover, Lemma \ref{lem:evens} demonstrates that for some $h$ satisfying $1 \leq h \leq \frac{j-i-1}{2}$, $b_{i+2h-1} = 0$. 
 
 Consider $S_q = \{x_{i+1}, x_{i+3}, \ldots, x_{i+2h-1}, x_{i+2h}, x_{i+2h+2}, \ldots,x_{j-1}\}$. We claim that $w_{S_q}(m_q) = b(m_q)$.
We see

\begin{align*}
m_q &= x_i^{a_i} e_{i}^{b_i} e_{i+1}^{b_{i+1}} \cdots e_{j-2}^{b_{j-2}}e_{j-1}^{b_{j-1}} x_{j}^{a_{j}}\\
&= x_i^{a_i} (x_{i} x_{i+1})^{b_i} (x_{i+1} x_{i+2})^{b_{i+1}} \cdots (x_{j-2}x_{j-1})^{b_{j-2}}(x_{j-1} x_j)^{b_{j-1}} x_{j}^{a_{j}}\\
&= x_i^{(a_i + b_i)} x_{i+1}^{(b_{i} + b_{i+1})} x_{i+2}^{(b_{i+1} + b_{i+2})} \cdots x_{j-1}^{(b_{j-2} + b_{j-1})} x_{j}^{(b_{j-1} + a_{j})}.
\end{align*}

Then:

\begin{align*}
w_{S_q}(m_q) &= (b_i+b_{i+1}) + (b_{i+2}+b_{i+3}) + \cdots + (b_{i+2h-2}+b_{i+2h-1}) + (b_{i+2h-1} + b_{i+2h}) \\
&\hspace{1cm} + (b_{i+2h+1}+b_{i+2h+2}) + \cdots + (b_{j-2}+b_{j-1})\\
&= b_{i+2h-1} +  \sum\limits_{h = i}^{j-1} b_h\\
&= b_{i+2h-1} + b(m_q) \\
&= 0 + b(m_q) \\
&= b(m_q).
\end{align*}

Hence, it does not matter whether $V_{H_q}$ has an odd or even number of vertices because $w_{S_q}(m_q) = b(m_q)$ regardless.

Now, since each $S_q$ covers its respective set of vertices, the union of all of these disjoint subcovers $S = \cup S_q$ is a vertex cover of $C_{2n+1}$. In addition, as each $S_q$ is completely disjoint from any other subgraph's cover, $w_{S}(m) = \sum w_{S_q}(m_q) =\sum b(m_q)$. As each $b(m_q)$ was the number of edges that existed in that induced subgraph representation, and no two subgraphs contained any of the same edges, $\sum b(m_q) = b(m)$, the total number of edges in an optimal factorization of $m$. 
That is, we have constructed a vertex cover $S$ such that $w_{S}(m) = b(m) < t$.
Thus, $m\notin (L(t))$, and therefore $I^t = (L(t))$.
%
\end{proof}

\begin{corollary}\label{thm:sorta}
Given $I$ and $(D(t))$ as above, $I^{(t)} = I^t + (D(t))$.
\end{corollary}

\begin{proof}
Apply Theorem \ref{lem:tl<it} to the equation $I^{(t)} = (L(t)) +(D(t))$. 
\end{proof}

Now that we have proved that $I^{(t)} = I^t + (D(t))$, we will use this result to carry out various computations related to the interplay between ordinary and symbolic powers.

We close this section with a brief remark on the proof of Theorem \ref{lem:it<tl}.
Specifically, it relies on the fact that $C_{2n+1}$ is a cycle, but not that $C_{2n+1}$ is an odd cycle.
However, we focus on the odd cycle case as even cycles are bipartite, and \cite{1994:simisvasconcelosvillareal} showed that if $I$ is the edge ideal of a bipartite graph, then $I^t = I^{(t)}$ for all $t \ge 1$.

\section{Applications to Ideal Containment Questions}

Given the edge ideal $I$ of an odd cycle $C_{2n+1}$, Corollary \ref{thm:sorta} describes a structural relationship between $I^{(t)}$ and $I^t$ given any $t \ge 1$.
In this section, we will exploit this relationship to establish the conjecture of \cite{2004:worthenelliswilson}.
We then will compute the resurgence of $I=I(C_{2n+1})$ and explore the symbolic defect of various powers of $I$. 

We will begin by examining $D(t)$. 

\begin{lem}\label{lem:nozero}
For a given monomial $x^{\underline{a}}$, if there exists some $i$ such that $a_i = 0$, then $x^{\underline{a}} \not \in (D(t))$.
\end{lem}

\begin{proof}
Although all graphs have many different minimal vertex covers, odd cycles have a vertex cover that 
includes any two adjacent vertices and alternating vertices thereafter.

Without loss of generality, consider $x^{\underline{a}}$ and suppose $a_1=0$. Two such minimal vertex covers that include $x_1$ are $\{x_1, x_2, x_4, x_6, \ldots, x_{2n}\}$ and $\{x_1, x_3, x_5, \ldots, x_{2n+1}\}$.

In order for $x^{\underline{a}}$ to be in $D(t)$, it must be true that $w_{V'}(x^{\underline{a}}) \geq t$. This means that  $a_1 + a_2 +a_4 + \cdots + a_{2n} \geq t$ and $a_1 + a_3 + \cdots +a_{2n+1} \geq t$. Adding the inequalities yields $a_1 + (a_1 +a_2 +a_3 +\cdots + a_{2n+1}) \geq 2t$. As $a_1 = 0$, we have $\sum\limits_{i=1}^{2n+1} a_i \geq 2t$, which contradicts the requirement that $\deg(x^{\underline{a}}) < 2t$. 
Hence, any monomial $x^{\underline{a}}$ with at least one exponent equal to 0 cannot be an element of $D(t)$. 
\end{proof}


\begin{lem}\label{lem:levels}
For a given monomial $x^{\underline{a}}$ in $D(t)$, if $\deg(x^{\underline{a}}) = 2t - k$, then $x^{\underline{a}}$ is divisible by $(x_1 x_2 \cdots x_{2n+1})^k$.
\end{lem}

\begin{proof}
Let $x^{\underline{a}} \in D(t)$ such that $\deg(x^{\underline{a}}) = 2t - k$, and suppose that $x^{\underline{a}}$ is not divisible by $(x_1 x_2 \cdots x_{2n+1})^k$. 
This means that there exists an $i_0$ such that $a_{i_0} < k$. Moreover, since $x^{\underline{a}} \in (D(t))$, we must have $a_j > 0$ for all $j$.

If ${i_0}$ is odd, consider minimal vertex covers 
\[
V_1 = \set{x_1, x_3,\ldots, x_{i_0},x_{i_0+1},x_{i_0+3},\ldots,x_{2n}}
\] 
and 
\[
V_2 = \set{x_2, x_4, \ldots, x_{i_0-1}, x_{i_0}, x_{i_0+2},x_{i_0+4},\ldots,x_{2n+1}}.
\]
If $i_0$ is even, use 
\[
V_1 = \set{x_1,x_3,\ldots,x_{i_0-1},x_{i_0},x_{i_0+2},\ldots,x_{2n}}
\]
and 
\[
V_2 = \set{x_2,x_4,\ldots,x_{i_0},x_{i_0+1},x_{i_0+3},\ldots,x_{2n+1}}.
\]
%

In order for $x^{\underline{a}}$ to be in $D(t)$, it must be true that $w_{V_j}(x^{\underline{a}}) \geq t$ for $j=1,2$. 
When $i_0$ is odd, this means that $a_1 + a_3 +\cdots + a_{i_0} + a_{i_0+1} + \cdots + a_{2n} \geq t$ and $a_2 + a_4 + \cdots + a_{i_0-1} + a_{i_0} + a_{i_0+2} +\cdots +a_{2n+1} \geq t$ (and similarly if $i_0$ is even). 
Combining these, we see $2t \le a_{i_0} + (a_1 +a_2 +a_3 +\cdots + a_{2n+1}) = a_{i_0} + \left(\sum_{s=1}^{2n+1} a_s \right) = 2t-k+a_{i_0} < 2t$, a contradiction.
%
%
%
%
%
\end{proof}


The following corollary partially answers \cite[Conjecture 15]{2004:worthenelliswilson} in the affirmative.
Note that this is a restatement of \cite[Theorem 4.13]{2017DaoEtAl}.

\begin{corollary}\label{cor:ordsymequal}
Let $I=I(C_{2n+1})$. Then $I^{(t)} = I^t$ for $1\leq t \leq n$.
\end{corollary}

\begin{proof}
%
If $m\in D(t)$, then $\deg(m) < 2t$.
Since there are $2n+1 > 2t$ variables, at least two of them must have exponents equal to 0, contradicting Lemma \ref{lem:nozero}.
Thus, $D(t) = \emptyset$.
%
\end{proof}

A recent paper of Galetto, Geramita, Shin, and Van Tuyl \cite{2016:galettogeramitavantuyl} introduced the notion of symbolic defect to measure the difference between the symbolic power $I^{(t)}$ and ordinary power $I^{t}$.
For a given $m$, the symbolic defect $\sdefect(I,m)$ is the number $\mu(m)$ of minimal generators $F_1, F_2, \ldots, F_{\mu(m)}$ such that $I^{(m)} = I^m + (F_1, F_2, \ldots, F_{\mu(m)})$.
Corollary \ref{cor:ordsymequal} thus implies that $\sdefect(I(C_{2n+1}),t) = 0$ for all $t$ satisfying $1\le t \le n$.

\begin{corollary}\label{cor:t=n+1}
Let $I=I(C_{2n+1})$. 
Then $\sdefect(I, n+1) = 1$. 
In particular, $I^{(n+1)} = I^{n+1} + (x_1x_2\cdots x_{2n+1})$.
\end{corollary}

\begin{proof}
	Let $m\in D(n+1)$.
	Then $2n+1 \le \deg(m) < 2n+2$, and by Lemma \ref{lem:nozero}, each of the $2n+1$ variables must divide $m$.
	Therefore, $m = x_1 x_2 \cdots x_{2n+1}$.
%
%
%
	Thus, $I^{(n+1)} = I^{n+1} + (x_1x_2\cdots x_{2n+1})$.
\end{proof}

Recall that, if $0\ne I\subsetneq S = k[x_1,x_2,\ldots,x_r]$ is a homogenous ideal, the minimal degree of $I$, denoted $\alpha(I)$, is the least degree of a nonzero polynomial in $I$.
In particular, if $I$ is an edge ideal, $\alpha(I) = 2$, and $\alpha(I^s) = 2s$ for any $s \ge 1$.
In general, if $\alpha(I^{(t)}) < \alpha(I^s)$, we may conclude that $I^{(t)}\not\subseteq I^s$, but the converse need not hold.
When $I = I(C_{2n+1})$, however, it does, as the next lemma demonstrates.

\begin{lem}\label{cor:mindegree}
Let $I$ be the edge ideal of an odd cycle. Then $\alpha(I^{(t)}) < \alpha(I^s)$ if and only if $I^{(t)} \not \subseteq I^s$.
\end{lem}
\begin{proof}
The forward direction is clear.
For the converse, suppose that $\alpha(I^{(t)}) \geq \alpha(I^s)$, and recall
\[
	I^{(t)} = (m \mid \text{ for all minimal vertex covers } V', \, w_{V'}(m) \geq t). 
\]
As $I^t \subseteq I^{(t)}$, we note that $2t = \alpha(I^t) \ge \alpha(I^{(t)}) \ge \alpha(I^s) = 2s$.
Thus, if $m\in I^{(t)}$, $w_{V'}(m) \ge t \ge s$ and $\deg(m) \ge \alpha(I^{(t)}) \ge \alpha(I^s) = 2s$, and we observe

\begin{align*}
	I^{(t)} &\subseteq (m \mid \deg(m) \ge 2s \text{ and for all minimal vertex covers } V', \, w_{V'}(m) \geq s)\\
	&= (L(s))\\
	&= I^s,
\end{align*}

which completes the proof.
\end{proof}

Despite providing a condition which guarantees containments of the form $I^{(t)}\subseteq I^s$, Lemma \ref{cor:mindegree} does not actually compute $\alpha(I^{(t)})$, which is more delicate to compute than computing $\alpha(I^s)$.
We next adapt Lemma \ref{lem:jdef} and the linear programming approach of \cite{2016:BocciMFO} to compute it.
In order to do so, we make the following definition.

\begin{definition}\label{def:minvertexcovermatrix}
Fix a list of minimal vertex covers $V_1,V_2,\ldots,V_{r}$ for $C_{2n+1}$ such that $|V_i| \leq |V_{i+1}|$.
We define the minimal vertex cover matrix $A = (a_{ij})$ to be the matrix of 0's and 1's defined by:
\begin{equation}\label{eq:minvertmat}
	a_{ij} = \begin{cases}
		0 & \text{if $x_j\notin V_i$}\\
		1 & \text{if $x_j\in V_i$}.
	\end{cases}
\end{equation}
\end{definition}

\begin{remark}
	Note the minimum cardinality for a minimal vertex cover of $C_{2n+1}$ is $n+1$; in fact, there are $2n+1$ minimal vertex covers of size $n+1$.
	As we have seen, there do exist minimal vertex covers of size greater than $n+1$.
	These covers will be accounted for in rows $2n+2$ and higher of the minimal vertex cover matrix $A$.
\end{remark}

We first seek a lower bound of $\alpha(I^{(t)})$ using  linear programming.
Let 
\[
	t = s(n+1) + d \text{, where } 0 \leq d \leq n.
\]

Consider the following linear program ($\star$), where $A$ is the minimal vertex cover matrix, 
\[
\mathbf{b} = \left(\begin{matrix} 1 \\ \vdots \\ 1 \end{matrix}\right), \text{ and } \mathbf{c} = \left(\begin{matrix} t \\ \vdots \\ t \end{matrix}\right):
\]

\begin{tabular}{rl}
		minimize & ${\bf b}^T{\bf y} $\\
		subject to & $A{\bf y} \geq {\bf c}$ and  ${\bf y} \geq {\bf 0}.$
	\end{tabular}\hspace{2cm}$(\dagger)$

\noindent We refer to $(\star)$ as the \emph{alpha program}, and observe that if $\mathbf{y}^*$ is the value which realizes $(\dagger)$, we have $\alpha(I^{(t)}) \geq \mathbf{b}^T \mathbf{y}^*$.

Consider the following partition of $A$: let $A'$ be the submatrix of $A$ consisting of the first $2n+1$ rows (and thus corresponding to the $2n+1$ minimal vertex covers which contain exactly $n+1$ vertices) and $B$ the matrix consisting of the remaining rows of $A$.
We thus create the following sub-program of $(\dagger)$,

\begin{tabular}{rl}
		minimize & ${\bf b}^T{\bf y} $\\
		subject to & $A'{\bf y} \geq {\bf c}$ and  ${\bf y} \geq {\bf 0}.$
	\end{tabular}\hspace{2cm}$(\ddagger)$

\begin{lem}\label{lem:alphasubprog}
	The value of $(\ddagger)$ is $\frac{(2n+1)t}{n+1}$.
\end{lem}
\begin{proof}
	We claim that
	\begin{equation}
		\mathbf{y}^* = \left( \begin{matrix} \frac{t}{n+1} \\ \frac{t}{n+1} \\ \vdots \\ \frac{t}{n+1}\end{matrix}\right),
	\end{equation}
	a $(2n+1)\times 1$ column vector, is a feasible solution to ($\ddagger$).
	Indeed, $A'\mathbf{y}^*$ is a column vector whose entries are all $t = s(n+1)+d$, satisfying the constraint of the LP. 
	In this case, $\mathbf{b}^T \mathbf{y}^* = \frac{(2n+1)t}{n+1}$.
	
	To show that this is the value of $(\ddagger)$, we make use of the fundamental theorem of linear programming by showing the existence of an $\mathbf{x}^*$ which produces the same value for the dual linear program:
	
		\begin{tabular}{rl}
maximize & ${\bf c}^T{\bf x} $\\
subject to & $(A')^T{\bf x} \leq {\bf b}$ and  ${\bf x} \geq {\bf 0}$.
\end{tabular}
\hspace{2cm}$(\star)$
	

	\noindent Specifically, let 
	\[
		\mathbf{x}^* = \left( \begin{matrix} \frac{1}{n+1} \\ \frac{1}{n+1} \\ \vdots \\ \frac{1}{n+1} \end{matrix} \right).
	\]
	As the rows of $(A')^T$ again have exactly $n+1$ 1's, we see $(A')^T \mathbf{x}^* \leq \mathbf{b}$ is satisfied, and it is straightforward to check that $\mathbf{c}^T \mathbf{x}^* = \mathbf{b}^T \mathbf{y}^* = \frac{(2n+1)t}{n+1}$.	
\end{proof}

\begin{lem}\label{lem:alphalowerbound}
	The value of $(\dagger)$ is bounded below by $\frac{(2n+1)t}{n+1}$.
\end{lem}

\begin{proof}
	Observe that $(\dagger)$ is obtained from $(\ddagger)$ by (possibly) introducing additional constraints.
	Thus, the value of $(\dagger)$ is at least the value of $(\ddagger)$, which is $\frac{(2n+1)t}{n+1}$.
\end{proof}

\begin{proposition}\label{prop:alphas}
	For all $t\ge 1$, $\alpha(I^{(t)}) = 2t - \left\lfloor \frac{t}{n+1}\right\rfloor$.
\end{proposition}

\begin{proof}
	Let $t = s(n+1)+d$, where $0\leq d\leq n$.
	By Lemma \ref{lem:alphalowerbound}, we see that $\alpha(I^{(t)})$ is bounded below by the value of $(\dagger)$, i.e., $\alpha(I^{(t)}) \geq \frac{(2n+1)t}{n+1} = \frac{(2n+1)(s(n+1)+d)}{n+1} = (2n+1)s + 2d - \frac{d}{n+1}$.	
	As $0 \leq \frac{d}{n+1} < 1$, it's enough to find an element of degree $(2n+1)s+2d$ in $I^{(t)}$.
	We claim that 
	\[
		m = x_1^{s+d} x_2^{s+d} x_3^s x_4^s \cdots x_{2n+1}^s
	\] 
	is such an element.
	Note that any minimal vertex cover $V'$ (and hence minimal prime of $I$) will contain one of $x_1$ and $x_2$, and at least $n-1$ (if it contains both $x_1$ and $x_2$) or $n$ (if it contains only one of $x_1$ and $x_2$) other vertices.
	
	In the former case, $w_{V'}(m) \ge 2(s+d) + s(n-1) = s(n+1) + 2d \geq t$, and so $m \in I^{(t)}$.
	In the latter case, $w_{V'}(m) \ge (s+d) + sn = s(n+1) + d = t$, and again we see $m\in I^{(t)}$.
	
	Thus, $\alpha(I^{(t)})$ is an integer satisfying $(2n+1)s + 2d - \frac{d}{n+1} \leq \alpha(I^{(t)}) \leq (2n+1)s + 2d$, whence $\alpha(I^{(t)}) = (2n+1)s + 2d = 2(n+1)s+2d-s = 2t-s - \left\lfloor \frac{d}{n+1}\right\rfloor = 2t - \left\lfloor s+\frac{d}{n+1}\right\rfloor = 2t - \left\lfloor \frac{t}{n+1}\right\rfloor$.
\end{proof}

Recall that, given a nontrivial homogeneous ideal $I\subseteq k[x_1,x_2,\ldots,x_{2n+1}]$, the resurgence of $I$, introduced in \cite{2010:bocciharbourne}, is the number $\rho(I) = \sup \setof{m/r}{I^{(m)}\not\subseteq I^r}$.

\begin{theorem}\label{thm:resurgence}
If $I=I(C_{2n+1})$, then $\rho(I) = \frac{2n+2}{2n+1}$.
\end{theorem}

\begin{proof}
Let $T = \setof{m/r}{I^{(m)}\not\subseteq I^r}$, and suppose that $I^{(m)} \not \subseteq I^r$. 
By Lemma \ref{cor:mindegree}, $\alpha(I^{(m)}) < \alpha(I^r)$. 
Since we know $\alpha(I^r) = 2r$ and $\alpha(I^{(m)}) = 2m - \lfloor \frac{m}{n+1} \rfloor$ by Proposition \ref{prop:alphas}, it follows that $2m - \lfloor \frac{m}{n+1} \rfloor < 2r$, and that $2m - \frac{m}{n+1} \leq 2m - \lfloor \frac{m}{n+1} \rfloor < 2r $. Thus $2m - \frac{m}{n+1}  < 2r$, and we conclude that
$\frac{m}{r} < \frac{2n+2}{2n+1}$.

%

\textbf{Claim:} If $m/r\in T$, then $(m+2n+2)/(r+2n+1)\in T$.

\textbf{Proof of Claim:} By Lemma \ref{cor:mindegree}, it is enough to show that $\alpha(I^{(m+2n+2)}) < \alpha(I^{r+2n+1})$. 
By Proposition \ref{prop:alphas}, we have

\begin{align*}
\alpha(I^{(m+2n+2)}) &= 2(m+2n+2) - \left\lfloor \frac{m+2n+2}{n+1}\right\rfloor\\
&= 2m+4n+4 - \left\lfloor \frac{m}{n+1} + \frac{2n+2}{n+1}\right\rfloor\\
&= 4n+2 + \left( 2m - \left\lfloor \frac{m}{n+1}\right\rfloor\right)\\
&= 4n+2 + \alpha(I^{(m)}) \\
&< 2(2n+1) + \alpha(I^r) \\
&= \alpha(I^{r+n+1}).
\end{align*}

Let $m_0 = r_0 = n+1$ and $a_0 = \frac{m_0}{r_0}$ and observe that $I^{(m_0)}\not\subseteq I^{r_0}$. 
Then recursively define  $a_{k} = \frac{m_k}{r_k}$ where $m_k =  m_{k-1}+2n+2$ and $r_k = r_{k-1}+2n+1$. 
By the claim above, $a_k = m_k/r_k \in T$. 
From this recursive definition, we obtain the explicit formula $a_{k} = \frac{n+1 + k(2n+2)}{n+1 + k(2n+1)}$, and conclude 
that $\rho(I) = \frac{2n+2}{2n+1}$.

\end{proof}


Recall that Corollaries \ref{cor:ordsymequal} and \ref{cor:t=n+1} imply, for $I = I(C_{2n+1})$, that
\[
	\sdefect(I,t) = \begin{cases} 0 & \text{if } t \le n \\ 1 & \text{if } t = n+1.\end{cases}
\]

Next, we explore additional terms in the symbolic defect sequence.
Our general approach is to rely on the decomposition described in Corollary \ref{thm:sorta}.
In the parlance of our work, the symbolic defect is the size of a minimal generating set for the ideal $(D(t))$.
When $n+2\le t\le 2n+1$, the elements of $D(t)$ have degree equal to $2t-1$ by Proposition \ref{prop:alphas}, so it suffices to count them.

\begin{lem}\label{lem:edgemonomialproducts}
	Let $t$ satisfy $n+2\le t\le 2n+1$.
	Then if $m\in D(t)$, $m/(x_1 x_2 \cdots x_{2n+1})$ is the product of exactly $t-n-1$ edge monomials.
\end{lem}

\begin{proof}
	Let $m\in D(t)$, 
	Thus, $m$ is divisible by the product of at most $t-1$ edge monomials.
	Since $\deg(m)=2t-1$ and $\alpha(I^{(t)}) = 2t - 1 > 2(t-1) = \alpha(I^{t-1})$ by Lemma \ref{cor:mindegree}, $m\in I^{t-1}$, and therefore $m$ is divisible by exactly $t-1$ edge monomials.
	Thus, an optimal factorization of $m$ is
	\[
		m = x_{i_0} e_1^{b_1} e_2^{b_2} \cdots e_{2n+1}^{b_{2n+1}}, \text{ where } \sum b_i = t-1.
	\]
	That is, $m$ has a single ancillary with exponent 1.

	Without loss of generality, assume that the ancillary of $m$ is $x_1$ (if it is not, an appropriate permutation of the indices and exponents can be applied to make it $x_1$).
	Write $m = x_1 e_1^{b_1} e_2^{b_2} \cdots e_{2n+1}^{b_{2n+1}}$ and $x_1 x_2 \cdots x_{2n+1} = x_1 e_2 e_4 \cdots e_{2n}$.
	Assume that $b_{2k} = 0$ for some $k\ge 1$, and consider the minimal vertex cover
	\[
		V = \set{x_2, x_4, \ldots, x_{2k}, x_{2k+1}, x_{2k+3}, \ldots, x_{2n+1}}.
	\]
	We observe that $w_V(m) = b_1 + b_2 + \cdots + b_{2k-1} + b_{2k} + b_{2k} + b_{2k+1} + \cdots + b_{2n+1} = b_{2k} + \sum b_i = 0 + \sum b_i = t-1$, which contradicts the assumption that $m\in D(t)$.
	
	Thus, $b_{2k} \ge 1$ for all $k\ge 1$, which allows us to write
	\[
		m = (x_1 e_2 e_4 \cdots e_{2n}) e_1^{b_1} e_2^{b_2-1} e_3^{b_3} \cdots e_{2n}^{b_{2n}-1} e_{2n+1}^{b_{2n+1}}.
	\]
	Finally, we see that
	\[
		m/(x_1 x_2\cdots x_{2n+1}) = e_1^{b_1} e_2^{b_2-1} e_3^{b_3} \cdots e_{2n}^{b_{2n}-1} e_{2n+1}^{b_{2n+1}}
	\]
	is the product of $t-n-1$ edge monomials.
\end{proof}

\begin{theorem}\label{thm:sdefects}
	Let $I = I(C_{2n+1})$.
	Then, for $t$ satisfying $n+2\le t \le 2n+1$, we have
	\[
		\sdefect(I,t) = \multichoose{2n+1}{t-n-1}, 
	\]
	where $\multichoose{2n+1}{t-n-1}$ denotes multichoose.
\end{theorem}

\begin{proof}
	As stated above, we wish to compute the size of a minimal generating set of $(D(t))$.
	For $n+2\le t\le 2n+1$, this means counting the number of monomials of degree $2t-1$ in $D(t)$.

	By Lemma \ref{lem:edgemonomialproducts}, the monomial $p = m/x_1 x_2 \cdots x_{2n+1}$ is the product of exactly $(t-1) - n$ edge monomials.
	 Thus, we may factor any $m\in D(t)$ as $m = x_1 x_2 \cdots x_{2n+1} p$, where $p$ is the product of exactly $t-n-1$ edge monomials.

	Conversely, suppose that $m = x_1 x_2 \cdots x_{2n+1} p$, where $p$ is the product of exactly $t-n-1$ edge monomials.
	Since $\deg(m) = 2t-1$ and, if $V$ is any minimal vertex cover of $C_{2n+1}$, we have $w_V(m) = w_V(x_1 x_2 \cdots x_{2n+1}) + w_V(p) \ge n+1 + t-n-1 = t$, where $w_V(p)\ge t-n-1$ follows from the fact that $p\in I^{t-n-1}$ by definition; thus, $m\in D(t)$.
	Therefore, to count the monomials in $D(t)$, it suffices to count all monomials $p$ that are products of $t-n-1$ edge monomials.

	We can visualize this problem by counting the number of ways to place these $t-n-1$ `edges' around the cycle, assuming that we can place multiple edges between the same two vertices. 
	By definition, this is
	\[
		\sdefect(I,t) = \multichoose{2n+1}{t-n-1}.
	\]	
	%
\end{proof}

In particular, 

\[
	\sdefect(I(C_{2n+1}),n+2) = 2n+1.
\]

\section{An additional containment question}\label{sec:future}
Our proof that $I^{(t)} = I^t + (D(t))$ does not hold for any graph other than a cycle, as it relies on the fact that each path between ancillaries is disjoint from every other path. This is not true in general.
This leads naturally to the following question.

\begin{question}\label{conj:anygraph}
	Let $G$ be a graph on the vertices $V = \set{x_1,x_2,\ldots,x_d}$ containing an odd cycle. 
	Suppose $I = I(G)$ is the edge ideal of $G$ in $R = k[x_1,x_2,\ldots,x_d]$, and let $L(t)$ and  $D(t)$ retain their usual definitions with respect to $G$.
	Does $I^{(t)} = I^t + (D(t))$ for all $t \ge 1$?
\end{question}

The following example answers Question \ref{conj:anygraph} in the negative.

\begin{ex}
	Consider the graph $G$ defined by $V(G) = \set{x_1, x_2, x_3, x_4, x_5, x_6, x_7}$ and $E(G) = \set{x_1 x_2, x_2 x_3, x_3 x_4, x_4 x_5, x_5 x_1, x_1 x_6, x_6 x_7}$, and let $m = x_1^2 x_2^2 x_3^2 x_4^2 x_5^2 x_7^2$.
	Observe that $m\notin I^6$, but as every minimal vertex cover $V$ of $G$ contains three of $x_1, x_2, x_3, x_4, x_5$, we have $w_V(m) \ge 2\cdot 3 = 6$.
	Thus, $I^6 \ne (L(6))$.
\end{ex}

However, we observe in the following two theorems that $I^t = (L(t))$ for certain classes of graphs.


One case in which Question \ref{conj:anygraph} holds is the case in which $G$ is an odd cycle with one additional vertex connected to exactly one vertex of the cycle. 

\begin{theorem}
Let $G$ be a graph consisting of $2n+2$ vertices and $2n+2$ edges such that $2n+1$ of them form a cycle and the remaining edge connects the remaining vertex to any existing vertex of the cycle. Further, let $I$ be the edge ideal of $G$ and let $L(t)$ and $D(t)$ retain their usual definitions with respect to $G$. Then $I^{(t)} = I^t + (D(t))$.
\end{theorem}
\begin{proof}
Without loss of generality, consider the cycle formed by $x_1, \ldots, x_{2n+1}$ with $e_{2n+2} = x_1 x_{2n+2}$ being the newly added edge. 

Let $m$ be a monomial expressed in optimal form $m = x_1^{a_1} x_2^{a_2} \cdots x_{2n+2}^{a_{2n+2}} e_1^{b_1} e_2^{b_2} \cdots e_{2n+2}^{b_{2n+2}}$, and recall that $b(m) = \sum b_i$.
As with the cycle, if $I^t = (L(t))$, it will follow that $I^{(t)} = I^t + (D(t))$.


By Lemma \ref{lem:it sub lt} we know that $I^t \subseteq (L(t))$ so we must only show the reverse containment. Let $m\not\in I^t$ (which implies that $b(m) < t$). 
Lemma \ref{lem:m12 not in lt} allows us to consider only cases where $m$ either has multiple ancillaries or has a single ancillary of degree at least 2.
We will construct a minimal vertex cover $V'$ of $G$ such that $w_{V'}(m) = b(m) < t$.

First, assume that $x_{2n+2}$ is the only ancillary of $m$, and observe that $a_{2n+2} \ge 2$. 
We may write $m$ as $m = x_{2n+2} e_1^{b_1} e_2^{b_2} \cdots e_{2n+2}^{b_{2n+2}} x_{2n+2}^{a_{2n+2} - 1}$. 
It cannot be true that $b_{i} \geq 1$ for all $i\in \{1,3,5,\ldots, 2n+1\}$, because it would then be possible to divide $m$ by some monomial $p = x_{2n+2} e_1 e_3\cdots e_{2n+1} x_{2n+2}$ which must be in optimal form by Lemma \ref{lem:splitm}; however, in this case, $p = e_{2n+2} e_2 e_4 \cdots e_{2n} e_{2n+2}$, contradicting that $p$ was in optimal form. 
Thus, at least one $b_{2j+1}$ is 0.
Then construct $V'$ as follows:


\begin{enumerate}
    \item If $b_1 = 0$, 
    let $$V' = \set{x_1, x_2, x_4, \ldots, x_{2n}}.$$ Then $w_{V'}(m) = (b_{2n+1}+b_{2n+2}+b_1) + (b_1+b_2) + (b_3+b_4) + \cdots + (b_{2n-1}+b_{2n}) = b_1 + \sum_{i=1}^{2n+2} b_i = 0 + \sum_{i=1}^{2n+2} b_i = b(m) < t$.
    \item If $b_{2j+1} = 0$ for some $j > 0$, let $$V' = \set{x_1, x_3, x_5, \ldots, x_{2j+1}, x_{2j+2}, x_{2j+4}, x_{2j+6}, \ldots, x_{2n}}.$$
     Then $w_{V'}(m) = (b_{2n+1}+b_{2n+2}+b_1) + (b_2+b_3) + \cdots + (x_{2j}+x_{2j+1}) + (x_{2j+1}+x_{2j+2}) + \cdots + (b_{2n-1}+b_{2n}) = b_{2j+1} + \sum_{i=1}^{2n+2} b_i = 0 + \sum_{i=1}^{2n+2} b_i = b(m) < t$.
\end{enumerate}

Now suppose that all ancillaries of $m$ are among the set $\set{x_1, x_2, \ldots, x_{2n+1}}$. 
By adapting the argument from Theorem \ref{lem:tl<it}, we may assume that there is either one ancillary with exponent at least 2, or that there are multiple ancillaries.
Use the construction in the proof of Theorem \ref{lem:tl<it} to decompose the subgraph $C_{2n+1}$ of $G$ as $H_1, H_2, \ldots, H_r$. 
Define $m_{C_{2n+1}} = x_1^{a_1} \cdots x_{2n+1}^{a_{2n+1}} e_1^{b_1} \cdots e_{2n+1}^{b_{2n+1}}$, i.e., $m_{C_{2n+1}} = m/(x_{2n+2}^{a_{2n+2}} e_{2n+2}^{b_{2n+2}})$.
The proof of Theorem \ref{lem:tl<it} provides minimal subcovers $S_1, S_2, \ldots, S_r$ such that $S = \cup S_q$ and $w_S(m_{C_{2n+1}}) = \sum_{i=1}^{2n+1} b_i$.

If $x_1\in S$, then $S$ covers $G$ and $w_S(m) = w_S(m_{C_{2n+1}}) + b_{2n+2} = \sum_{i=1}^{2n+2} b_i = b(m) < t$.
In this case, we may let $V' = S$.

On the other hand, if $x_1\notin S$, let $V' = S\cup \set{x_{2n+2}}$.
Then $w_{V'}(m) = w_S(m) + b_{2n+2} = \sum_{i=1}^{2n+2} b_i = b(m) < t$.

Next, assume that the ancillaries of $m$ are $x_{2n+2}$ and at least one $x_j$ in the cycle (where $j\ne 1$; if $j= 1$, we may write $x_{2n+2} x_1 = e_{2n+2}$, contradicting the assumption that $m$ is in optimal form). 
Use the construction of Theorem \ref{lem:tl<it} to decompose the cycle into subgraphs $H_1, H_2, \ldots, H_r$ and note that $x_1$ is a vertex in $H_r$. 
Observe  that since $x_1$ is not ancillary, $x_1\notin H_i$ for any $i\ne r$.
Let the vertices of $H_i$ be represented by $\set{x_{\ell_{i}}, x_{\ell_{i}+1}, \ldots, x_{\ell_{i+1}}}$, where $x_{\ell_1}, \ldots, x_{\ell_r}$ are ancillaries, and we wrap around with $x_{\ell_{r+1}}$ representing $x_{\ell_1}$.
For all $i \ne r$, the proof of Theorem \ref{lem:tl<it} gives a construction of a minimal vertex subcover $S_i$ with the required properties.
Now construct a subgraph $H_r'$ of $G$ as follows: $V(H_r') = V(H_r) \cup \set{x_{2n+2}}$ and $E(H_r') = E_{H_r} \cup \set{\set{x_1, x_{2n+2}}}$.
Decompose $H_r'$ as two induced subgraphs $H_{r_a}'$ and $H_{r_b}'$ of $G$ on the vertices $\set{x_{\ell_r}, \ldots, x_1, x_{2n+2}}$ and $\set{x_{2n+2}, x_1, \ldots, x_{\ell_1}}$.
We observe that we may now use the construction in the proof of Theorem \ref{lem:tl<it} to build minimal covers of $H_{r_a}'$ and $H_{r_b}'$ containing $x_1$ (and not $x_{2n+2}$) whose union gives a cover $S_r$ of $H_r'$.
Given $m_r = x_{\ell_r}^{a_{\ell_r}} x_{\ell_{1}}^{a_{\ell_{1}}} x_{2n+2}^{a_{2n+2}} e_{\ell_r}^{b_{\ell_r}} e_{\ell_r+1}^{b_{\ell_r + 1}} \cdots e_{2n+2}^{b_{2n+2}} e_1^{b_1} \cdots e_{\ell_1-1}^{b_{\ell_1-1}}$, note that $w_{S_r'}(m_r) = b(m_r)$.
Then the union $V' = \cup_{i=1}^r S_i$ has the required property that $w_{V'}(m) = b(m)$.

In all cases, $m\notin L(t)$, and, by extension, $m\notin (L(t))$.

\end{proof}

We also verify that the answer to Question \ref{conj:anygraph} is positive when $G$ is a complete graph.
Thus, additional study is needed to identify the precise graph-theoretic property for which Question \ref{conj:anygraph} has an affirmative answer.

\begin{theorem}
  Let $R=k[x_1,\ldots,x_n]$ and let $K_n$ denote the complete graph on $\set{x_1,\ldots,x_n}$. 
  Further, let $I=I(K_n)$ and $L(t)$ and $D(t)$ maintain their definitions as above. Then $I^{(t)} = I^t + (D(t))$
\end{theorem}

\begin{proof}
  Let $e_{i,j}$ denote the edge between $x_i$ and $x_j$ such that $i<j$. 
  We will show that $I^t = (L(t))$.
  By Lemma \ref{lem:it sub lt}, we must only show $(L(t))\subseteq I^t$.
  Let $m\not\in I^t$ (which implies that $b(m) < t$), and recall that Lemma \ref{lem:m12 not in lt} allows us to consider only cases where $m$ either has multiple ancillaries or has a single ancillary of at least degree 2. 
  Let $m=x_1^{a_1}\cdots x_n^{a_n}e_{1,2}^{b_{1,2}}\cdots e_{n-1,n}^{b_{n-1,n}}$ be in optimal form. Then $m$ has at most $1$ ancillary because if $x_i^{a_i}$ and $x_j^{a_j}$ were both ancillaries, then $m$ could be expressed as $$m=x_1^{a_1}\cdots x_i^{a_i-1}\cdots x_j^{a_j-1}\cdots x_n^{a_n}e_{1,2}^{b_{1,2}}e_{1,3}^{b_{1,3}}\cdots e_{i,j}^{b_{i,j}+1}\cdots e_{n-1,n}^{b_{n-1,n}}.$$ 
  Thus $m$ has exactly $1$ ancillary and it must have a degree of at least 2.
  
  Without loss of generality, let $x_1^{a_1}$ be the ancillary of $m$. Note that $b_{i,j}=0$ if $i,j\not=1$. If this was not the case, $m$ could be expressed in a more optimal form as $$m=x_1^{a_1-2}\cdots x_n^{a_n}e_{1,2}^{b_{1,2}}e_{1,3}^{b_{1,3}}\cdots e_{1,i}^{b_{1,i}+1}\cdots e_{1,j}^{b_{1,j}+1}\cdots e_{i,j}^{b_{i,j}-1}\cdots e_{n-1,n}^{b_{n-1,n}}.$$
  Let $V' = \set{x_2,\ldots, x_n}$. 
  Observe that $V'$ covers $K_n$ and
\begin{align*}
  w_{v'}(m) &= w_{v'}(x_2^{b_{1,2}+b_{2,3}+b_{2,4}+\cdots+b_{2,n}} + x_3^{b_{1,3}+b_{2,3}+b_{3,4}+\cdots+b_{3,n}} + \cdots + x_n^{b_{1,n}+b_{2,n}+b_{3,n}+\cdots+b_{n-1,n}}) \\
  &= w_{v'}(x_2^{b_{1,2}+0+\cdots+0} + x_3^{b_{1,3}+0+\cdots+0} + \cdots + x_n^{b_{1,n}+0+\cdots+0}) \\
  &= \sum_{j=2}^n b_{1,j} \\
  &= \sum_{j=2}^n b_{1,j} + \sum_{i=2}^{n-1}\sum_{j=i+1}^n b_{i,j} \\
  &= \sum_{i=1}^{n-1}\sum_{j=i+1}^n b_{i,j} \\
  &= b(m).
\end{align*}
Thus $w_{v'}(m) = b(m) < t$, so $m \not\in L(t)$ and by the same argument, no divisor of $m$ is in $L(t)$, which means $m \not\in (L(t))$. 

Therefore, $I^t = (L(t))$. 
Because $I^{(t)} = (L(t)) + (D(t))$ by Corollary \ref{thm:sorta}, this leads to the desired result that $I^{(t)} = I^t + (D(t))$.
\end{proof}

\subsection*{Acknowledgements} This work was supported by Dordt College's summer undergraduate research program in the summer of 2017. All three authors wish to express deep gratitude to the Dordt College Office of Research and Scholarship for the opportunity to undertake this project.
The authors also wish to thank the referee for the many helpful comments on their work. In particular, the proof of Theorem \ref{thm:sdefects} was improved substantially by the referee's suggestions.

\bibliography{summer17.bib}{}
\bibliographystyle{plain}

\end{document}